\documentclass[final,leqno,onefignum]{amsart}
\usepackage{t1enc,amsfonts,latexsym,amsmath,verbatim,amssymb,epsfig,graphics,psfrag,mathdots}

\def\B{{\mathcal B}}

\def\K{{\mathcal K}}
\def\H{{\mathcal H}}

\def\n{{\mathcal N}}
\def\r{{\mathcal R}}

\def\N{{\mathbb N}}
\def\m{{\mathbb M}}
\def\C{{\mathbb C}}

\def\R{{\mathbb R}}

\def\d{{\mathcal D}}

\def\M{{\mathcal M}}

\def\e{{\mathcal E}}

\def\f{{\mathcal F}}

\newcommand{\Span}{\mathsf{Span}~ }

\newcommand{\Ran}{\mathsf{Ran}~ }

\newcommand{\dist}{\mathsf{dist} }

\numberwithin{equation}{section}

\theoremstyle{plain}
\newtheorem{theorem}{Theorem}[section]
\newtheorem{proposition}[theorem]{Proposition}
\newtheorem{lemma}[theorem]{Lemma}

\newtheorem{definition}  [theorem] {Definition}

\theoremstyle{definition}
\newtheorem{ex}[theorem] {Example}
\begin{document}

\title{Alternating projections on non-tangential manifolds.}
\author{Fredrik Andersson}
\address{Centre for Mathematical Sciences, Lund University, Sweden}
\email{fa@maths.lth.se}

\author{Marcus Carlsson}
\address{Departamento de Matem\'{a}ticas,Universidad de Santiago de Chile, Chile}
\email{marcus.carlsson@usach.cl}


\date{}

\dedicatory{}

\begin{abstract}
We consider sequences
$(B_k)_{k=0}^\infty$ of points obtained by projecting back and forth
between two manifolds $\M_1$ and $\M_2$, and give conditions
guaranteeing that the sequence converge to a limit
$B_\infty\in\M_1\cap\M_2$. Our motivation is the study of algorithms
based on finding the limit of such sequences, which have proven
useful in a number of areas. The intersection is
typically a set with desirable properties, but for which there is no
efficient method of finding the closest point $B_{opt}$ in
$\M_1\cap\M_2$. We prove not only that the sequence of alternating
projections converges, but that the limit point is fairly close to
$B_{opt}$, in a manner relative to the distance $\|B_0-B_{opt}\|$,
thereby significantly improving earlier results in the field. A
concrete example with applications to frequency estimation of
signals is also presented.
\end{abstract}

\maketitle

\section{Introduction}
Let $\K$ be a finite dimensional Hilbert space over $\R$ and let
$\M_1, ~\M_2\subset\K$ be manifolds. Suppose that for any $B\in\K$
the closest point on $\M_j$, $j=1,2$, is well defined and lets
denote it by $\pi_j(B)$. Let the corresponding projection onto the
intersection $\M_1\cap\M_2$ be denoted by $\pi(B)$. Suppose that we
are interested in finding the closest point on $\M_1\cap\M_2$ and
that the ``projection operators'' $\pi_1$ and $\pi_2$ can be
efficiently computed, whereas $\pi$ can not. The issue treated here
is how to use $\pi_1$ and $\pi_2$ to obtain an approximation of
$\pi$. A classical result by von Neumann \cite{vonNeumann} says that if $\M_1$ and
$\M_2$ are affine linear manifolds, then the sequence of alternating
projections
\begin{equation}\label{a1}\pi_1(B),~\pi_2(\pi_1(B)),~\pi_1(\pi_2(\pi_1(B))),~\pi_2(\pi_1(\pi_2(\pi_1(B)))),\ldots\end{equation}
converges to $\pi(B)$.  Moreover, the convergence rate is
determined by the angle between $\M_1$ and $\M_2$. This paper is
concerned with extensions of this result to non-linear manifolds.

Hence, given $B\in\K$, let $B_1=\pi_1(B)$ and
\begin{equation}\label{eq44}B_{k+1}=\left\{\begin{array}{cc}
                                    \pi_1(B_k) & k, \text{ is even,} \\
                                    \pi_2(B_k) & k, \text{ is odd.}
                                  \end{array}\right.
\end{equation} In contrast to the case where $\M_j$ is an affine linear manifold, $B_\infty \ne \pi(B)$. However, given that $\M_j$ behave nicely, we may expect $\B_\infty \approx \pi(B)$.

Alternating projection schemes of this kind have been used in a
number of applications, cf. \cite{cadzow, Grigoriadis, Levi:83,Liu_altproj_hankel,Lu_altproj_hankel,  Marks, structured_low_rank_survey, Prabhu_altproj_hankel}. For
instance, $\K$ can be the set $\m_{m,n}$ of $m\times n$-matrices,
and the manifolds $\M_j$ are subsets with a certain structure, e.g.
matrices with a certain rank, self-adjoint matrices, Hankel or
Toeplitz matrices etc. Alternating projection schemes between
several linear subsets in the infinite dimensional setting was
recently investigated in \cite{bad}. Much emphasis has been put towards the usage of alternating projections for the case of complex manifolds, see for instance \cite{Bauschke93, Bauschke96}. Connections to the EM algorithm are given in \cite{Bauschke99}.


In \cite{cadzow}, Zangwill's Global Convergence Theorem
\cite{zangwill} is used to motivate the convergence of the
alternating projection scheme above. From Zangwill's theorem it is
possible to deduce that if the sequence $(B_k)_{k=1}^\infty$ is
bounded and the distance to $\M_1\cap\M_2$ is strictly decreasing,
then $(B_k)_{k=1}^\infty$ has a convergent subsequence to a point
$B_{\infty}\in\M_1\cap\M_2$, i.e., a result related to the fact that
any bounded sequence in a compact set has a convergent subsequence.
Thus, the use of Zangwill's theorem in this context does not provide
any information about whether the limit point $B_{\infty}$ exists,
or if so, whether it is close to $\pi(B)$.

Recently, A. Lewis and J. Malick presented stronger results, valid
under more restrictive conditions on $\M_1$ and $\M_2$. Before
discussing their results in more detail, we give two simple examples
that illustrate some of the difficulties that may arise when using
alternating projections on non-linear manifolds.

\begin{ex}\label{ex1}
Let $\K=\R^2$ and set $\M_1=\{(t,(t+1)(3-t)/4):t\in\R\}$ and
$\M_2=\R\times\{0\}$. It is easily seen that
$\pi_1((1,0))=(1,1)$ and $\pi_2((1,1))=(1,0)$, and hence the
sequence of alternating projections does not converge, cf. Figure
\ref{fig1}. On the other hand, if
$(1+\epsilon,0)\in \M_2$, $\epsilon>0$, is used as a starting point, the sequence of alternating projections will converge to $(3,0)\in \M_1\cap\M_2$.
\end{ex}
\begin{figure}[tbh!]
\unitlength=1in
\begin{center}
\psfrag{10}[][l]{$(1,0)$}
\psfrag{11}[][l]{$(1,1)$}
\psfrag{30}[][l]{$(3,0)$}
\includegraphics[width=3.5in]{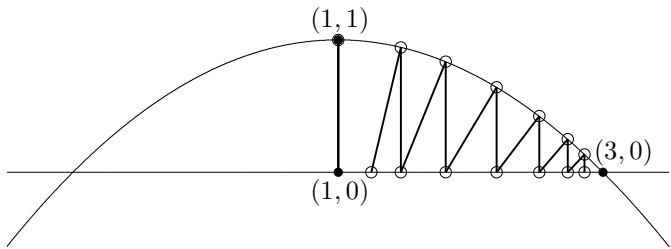}
\end{center}
\caption{\label{fig1} An example demonstrating that the algorithm can get stuck
in loops where even no subsequence converges to a point in the
intersection.}
\end{figure}

In general, it seems reasonable to assume that if the starting point is sufficiently close to an intersection point, then the sequence
does converge to a point in the intersection. The next example shows
that this is not the case of limited smoothness.

\begin{ex}\label{ex2}
Without going in to the details of the construction, we note that
one can construct a $C^1$-function $f$ such that, with
$\M_1=\R\times\{0\}$ and $\M_2=\{(t,f(t)):t\in\R\}$, the sequence of
alternating projections can get stuck in projecting back and forth
between the same two points. Figure \ref{fig11} explains the idea.
\end{ex}
\begin{figure}[tbh!]
\unitlength=1in
\begin{center}
\psfrag{10}[][l]{$(1,0)$}
\psfrag{11}[][l]{$(1,1)$}
\psfrag{30}[][l]{$(3,0)$}
\includegraphics[width=3.5in]{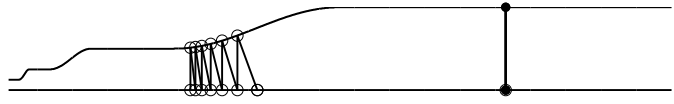}
\end{center}
\caption{\label{fig11}Alternating projections stuck in a loop.}
\end{figure}

However, if we have $f\in C^1$ and $f'(0)\neq 0$, it is hard to
imagine how to make a similar construction work. This is indeed
impossible if we assume additional smoothness, which follows both
from the present paper and \cite{lewis_malick}. In the terminology
of the latter, the condition $f'(0)\neq 0$ implies that $\M_1$ and
$\M_2$ are \textit{transversal} at $(0,0)$. In general, given
$C^1$-manifolds $\M_1,~\M_2$ and a point $A\in\M_1\cap\M_2$, we say
that $A$ is \textit{transversal} if
\begin{equation}\label{transversal}T_{\M_1}(A)+T_{\M_2}(A)=\K,\end{equation}
where $T_{\M_j}(A)$ denotes the tangent-space of $\M_j$ at $A$,
$j=1,2.$ The main result in \cite{lewis_malick} is roughly the
following:
\begin{theorem}\label{tLM}
Let $\M_1$ and $\M_2$ be $C^3$-manifolds and let $A\in\M_1\cap\M_2$
be transversal. If $B$ is close enough to $A$, then the sequence of
alternating projections $(B_k)_{k=1}^\infty$ given by (\ref{eq44})
converges to a point $B_\infty$ in $\M_1\cap\M_2$. Moreover,
$$\|B_{\infty}-\pi(B)\|\leq 2\|A-B\|.$$
\end{theorem}
The improvement over Zangwill's theorem is thus that the entire
sequence converges, and that the assumption of boundedness no longer
is necessary. Moreover, the limit $B_\infty$ is not too far off from
$\pi(B)$, (although in relative terms, i.e. comparing with
$\dist(B,\M_1\cap\M_2)=\|B-\pi(B)\|$, it need not be particularly
close either).

On the other hand, the assumption of transversalsality is rather restrictive. To demonstrate the essence of the transversality
assumption, we now present two cases which are not covered by the
above result, but where the conclusion still holds.

\begin{figure}[tbh!]
\unitlength=1in
\begin{center}
\psfrag{00}[][l]{$(0,0)$} \psfrag{11}[][l]{$\hspace{10pt}(1,1)$}
\includegraphics[width=2.5in]{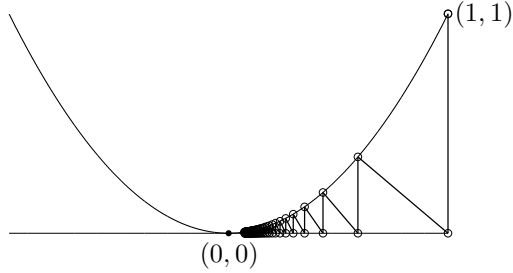}
\end{center}
\caption{\label{fig3} Lack of transversality due to tangential curves.}
\end{figure}

\begin{ex}\label{ex3}
With $\M_1=\R\times\{0\}$ and $\M_2=\{(t,t^2):t\in\R\}$,
transversality is not satisfied at $A=(0,0)$ (since
$T_{\M_1}(A)=T_{\M_2}(A)=\M_1$), but it is not hard to see that the
sequence of alternating projections still converges to $(0,0)$, (see
Figure \ref{fig3}).
\end{ex}

\begin{figure}[tbh!]
\unitlength=1in
\begin{center}
\psfrag{M1}[][l]{$\M_1$} \psfrag{M2}[][l]{$\M_2$}
\includegraphics[width=3.5in]{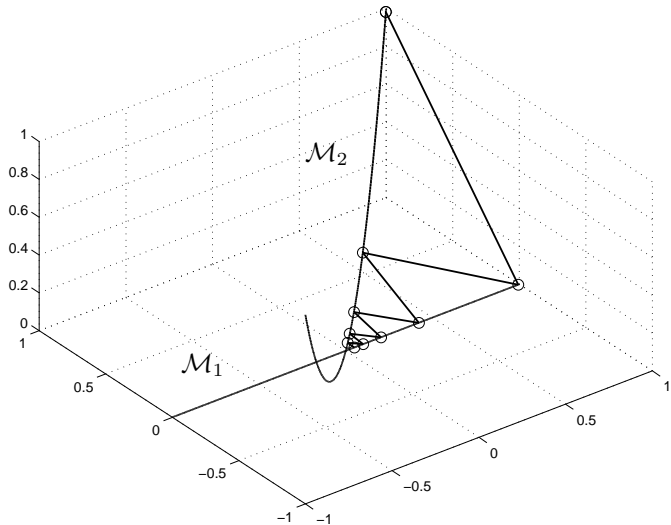}
\end{center}
\caption{\label{fig4}Lack of transversality due to too low
dimensionality.}
\end{figure}

The difference between Example \ref{ex2} and \ref{ex3} is that in
the latter case the manifolds are more regular.

\begin{ex}\label{ex4}
With $\K=\R^3$ and $\M_1=\R\times\{0\}^2$ and
$\M_2=\{(t,t,t^2):t\in\R\}$, transversality is not satisfied at
$A=(0,0)$, but again it seems plausible that the sequence of
alternating projections converges to $(0,0)$. See Figure \ref{fig4}.
\end{ex}
In Example \ref{ex4}, the two manifolds clearly sit at a positive
angle, but this situation is not covered by Theorem \ref{tLM}, since
the manifolds are of too low dimension to satisfy the transversality
assumption. In fact, if $\K$ has dimension $n$ and $\M_j$ has
dimension $m_j$, $j=1,2$, the transversality (\ref{transversal}) can
never be satisfied if $m_1+m_2<n$. But for many applications of
practical interest, one has $m_1+m_2<<n$. We will introduce a
concept which we call \textit{non-tangential}, which loosely
speaking says that the manifolds should have a positive angle in
directions perpendicular to $\M_1\cap\M_2$. The point $(0,0)$ in
Example \ref{ex4} is non-tangential, whereas the same point in
Example \ref{ex3} is not. A proper definition of non-tangentiality
is given in Definition \ref{def st}. A simplified version of our
main result is given in what follows.

\begin{theorem}\label{t2} Given a non-tangential point $A\in
\M_1\cap\M_2$ there exists an $s>0$ such that the sequence of
alternating projections (\ref{eq44}) converges to a point
$B_\infty\in\M_1\cap\M_2$, given that $\|B-A\|<s$. Moreover, given
any $\epsilon>0$ one can take $s$ such that
$$\|B_\infty-\pi(B)\|<\epsilon\|B-\pi(B)\|.$$
\end{theorem}
The full version of the main theorem is given in Section \ref{ap}.

\begin{figure}[ht]
\centering \psfrag{Bi}[c]{$B_\infty$} \psfrag{piB}[c]{$\pi(B)$}
\psfrag{A}[c]{$A$} \psfrag{ed}[c]{$< \epsilon d$} \psfrag{d}[c]{$d$}
\psfrag{s}[c]{$s$} \psfrag{BB}[c]{$B$}
\psfrag{M1M2}[c]{$\mathcal{M}_1 \cap \mathcal{M}_2$}
\includegraphics[width=.5\linewidth]{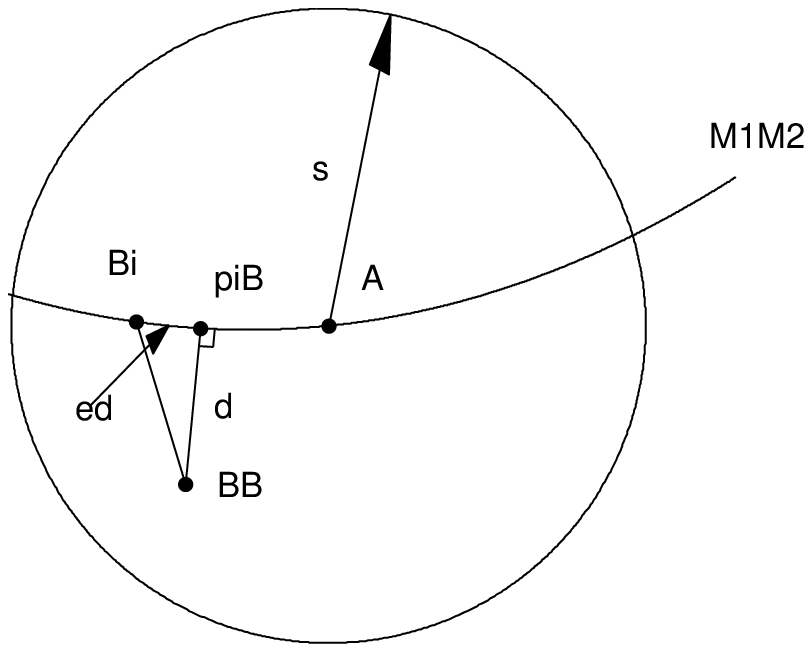}
\end{figure}

The improvement over Theorem \ref{tLM} mainly consists of two items.
Primarily, the assumption that the surfaces be non-tangential is not
at all restrictive, and in particular there is no implication on the
dimensions of $\M_1$ and $\M_2$. For the applications we are aware
of, the set of tangential points is very small, if it exists at all.
Secondly, as has been highlighted before, we are usually interested
not just in any point of $\M_1\cap \M_2$, but the closest point
$\pi(B)$. Here the theorem says that in relative terms, i.e. after
dividing with the distance to $\M_1\cap\M_2$, the error is small if
the distance to $\M_1\cap\M_2$ is small. It is also worth mentioning
that we only assume that the manifolds are $C^2$, although the
manifolds are $C^\infty$ in the applications that we are aware of.

In order to motivate the need for certain technicalities in the
abstract setting, as well as the need for a theorem concerning
low-dimensional manifolds, we will develop the theory in parallel
with an example, namely that when $\K$ is the space $\m_{n,n}(\C)$
of complex $n\times n$-matrices, $\M_1$ is the set of Hankel
matrices and $\M_2$ the set of matrices of rank at most $k$. The
former is a subspace of real dimension $2(n-1)$, whereas the latter
is not actually a manifold. However, it is ``locally'' a manifold of
real dimension $2(2nk-k^2)$ at all matrices $A$ with rank precisely
$k$, (which clearly constitutes the majority of matrices in the
set). If $k$ is small and $A$ is a non-tangential intersection
point, then $\dim(T_{\M_1})+\dim(T_{\M_2})=2(2nk-k^2+n-1)$, which is
much less than $\dim \K=2 n^2$, unless $k$ is close to $n$. Another
thing that does not match between this example and the theory
outlined above is that the projection $\pi_2$ is not uniquely
defined at certain points. This is common in algorithmic theory and
can be dealt with by using point-to-set maps, following
Zangwill \cite{zangwill}. However, in our case, this is not
necessary since we show that the projections are locally well
defined near non-tangential points, (Proposition \ref{p2}).
Moreover, from a practical perspective this is unnecessary since such
points are rarely encountered in applications.

For the above example, the interest in $\M_1\cap\M_2$ lies in the
fact that such sequences are samplings of functions which are sums
of $k$ exponential functions. Given a function $f$ on an interval
the alternating projections algorithm can then be used to find
approximations of $f$ by sums of $k$ exponential functions, which is
a problem of great practical interest. We demonstrate the idea in Figure \ref{fig_signal}. More thorough examples are conducted in \cite{freqest_altproj}.
\begin{figure}[tbh!]
\unitlength=1in
\begin{center}
\psfrag{00}[][l]{$(0,0)$} \psfrag{11}[][l]{$(1,1)$}
\includegraphics[width=2in]{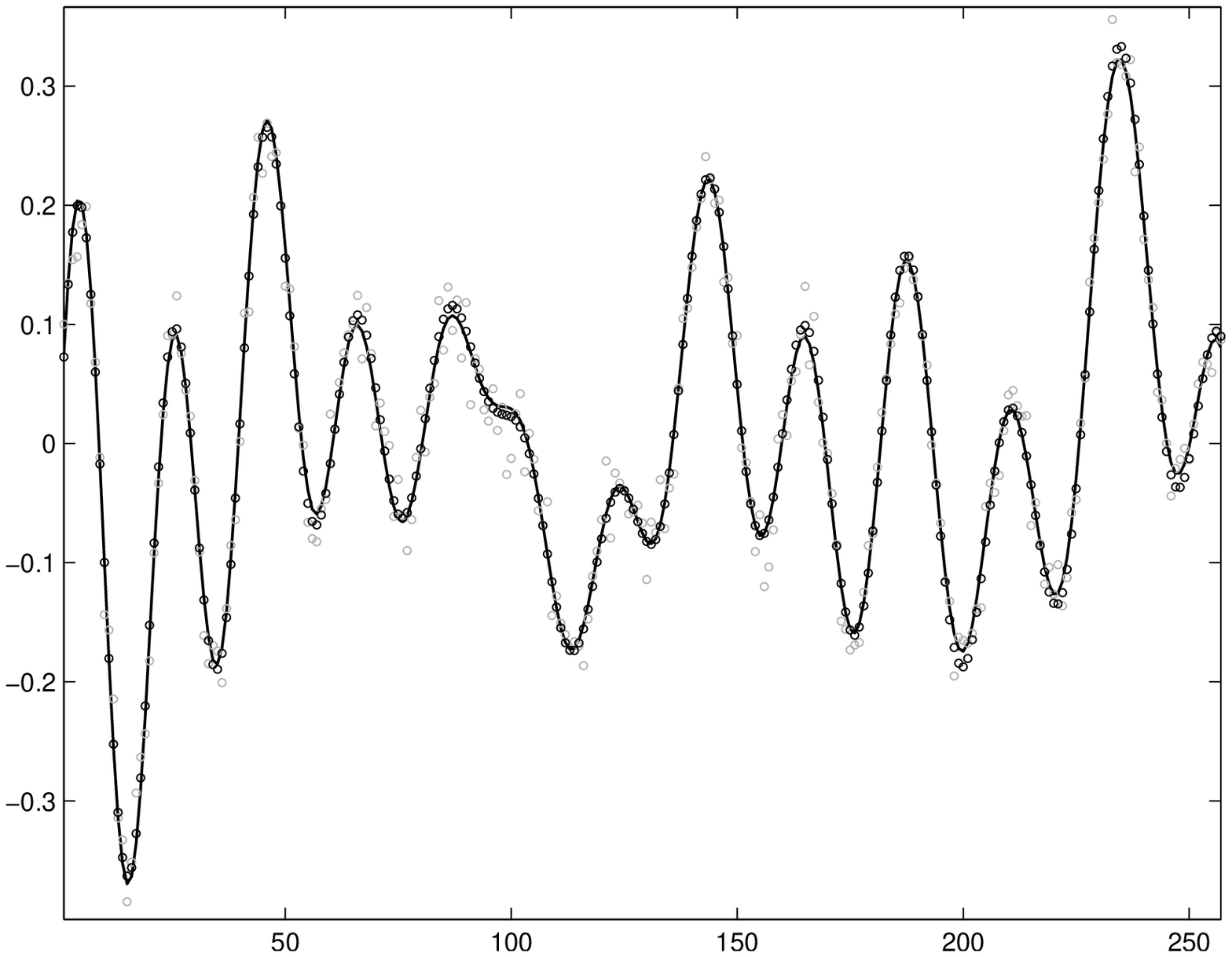}
\includegraphics[width=2in]{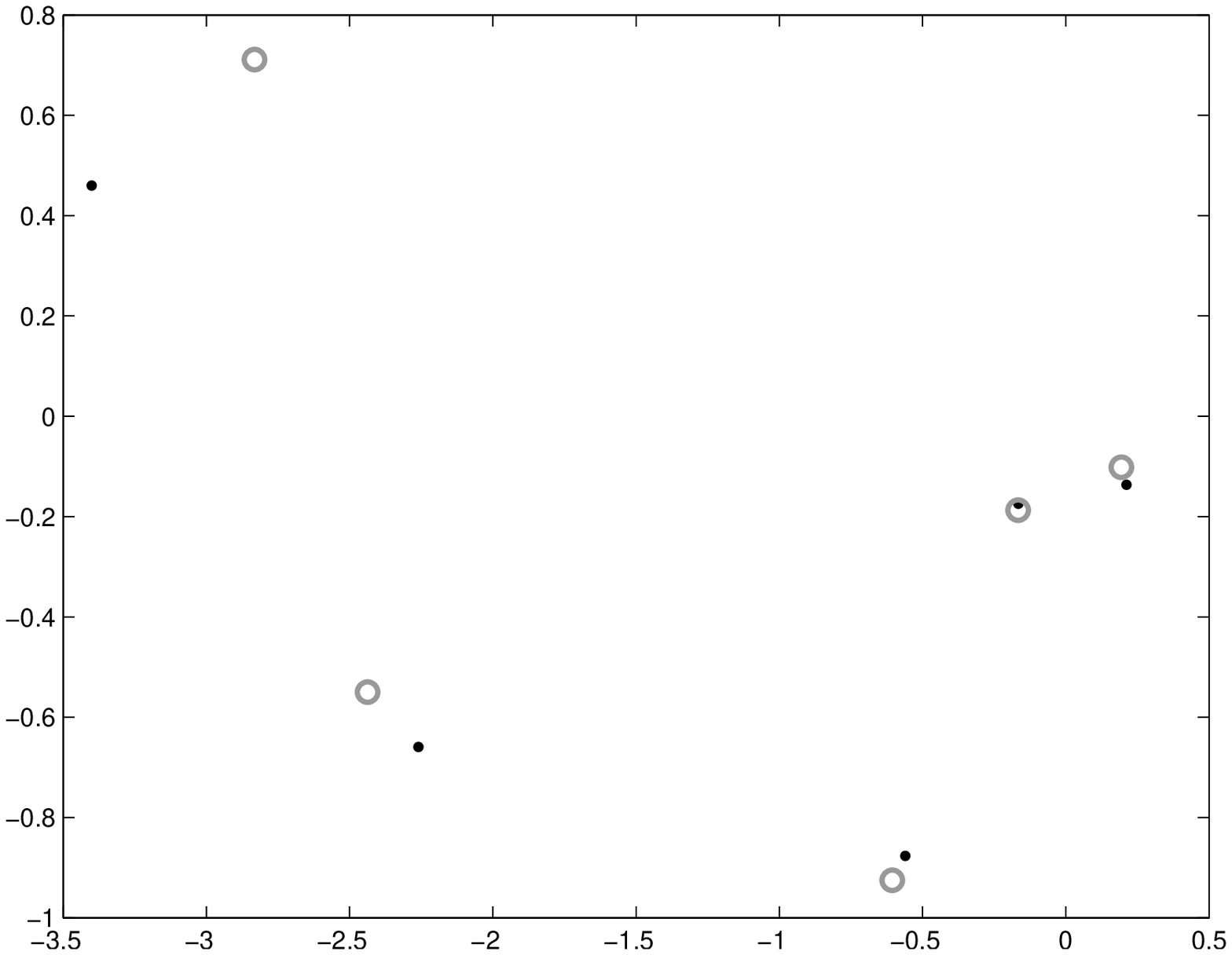}
\end{center}
\caption{Signal with five exponentials. Left panel: original signal
(black solid line); signal with white noise (gray circles); and the reconstruction
using the alternating Hankel projection method (black circes). Right panel:
original exponential nodes (black dots), estimated nodes from noisy
signal (gray circles.)\label{fig_signal}}
\end{figure}

\section{Case study; rank $k$ matrices versus Hankel matrices}\label{I}

We include this section for the reader to get the picture of a
typical application. However, this particular application is treated
in detail in \cite{freqest_altproj}, and therefore we will be very brief
here. In the setting considered by Lewis and Malick, the fact that
$\M_1\cap\M_2$ is itself a manifold (around a point $A$, say) follows
by the tranversality assumption (\ref{transversal}) and standard
differential geometry. In our (more general) setting, $\M_1\cap\M_2$
may fail to be a manifold. However, in the example developed in this
section, $\M_1\cap\M_2$ will be a manifold locally. In fact, it is
not hard to see that it will always happen for algebraic manifolds,
which is the case for all of the various applications presented in
\cite{cadzow}.

Let $n\in\N$ be fixed, let $\K=\m_{n,n}(\C)$ and let $\H\subset\K$
be the set of Hankel matrices, e.g., matrices of the form
\begin{equation}\label{hankel}\left(
                                            \begin{array}{ccccc}
                                              a_1 & a_2 & a_3 & \cdots & a_n \\
                                              a_2 & a_3 & \iddots & a_n & a_{n+1} \\
                                              a_3 & \iddots & \iddots & \iddots & \vdots \\
                                              \vdots & a_n & \iddots & \iddots & a_{2n-2} \\
                                              a_n & a_{n+1} & \cdots & a_{2n-2} & a_{2n-1} \\
                                            \end{array}
                                          \right)
\end{equation}
$\H$ is a linear subspace and, hence,  a manifold at each point, of
(real) dimension $2(2n-1)$. Denoting the matrix in (\ref{hankel}) by
$H(a)$, where $a=(a_1,\ldots,a_{2n-1})$, we obtain a natural chart
for $\H$, (by identifying $\C$ with $\R^2$ in the obvious way).

Given $k<n$, we denote by $\r_k\subset\K$ the set of matrices of
rank less than or equal to $k$. We wish to do alternating
projections as outlined in the introduction between $\H$ and $\r_k$,
which is slightly complicated by the fact that $\r_k$ is not a
manifold. However, it turns out that $\r_k$ is locally a manifold of
(real) dimension $2(2nk-k^2)$, apart from at some exceptional
points. More precisely, suppose $A\in\r_k$, and use the singular
value decomposition of $A$ to find $\sigma_A\in(\R^+)^k$ and unitary
matrices $U_A,V_A\in\m_{n,k}$ such that
\begin{equation}\label{parark}A=V_A\left(
                                   \begin{array}{cccc}
                                     \sigma_{A,1} & 0 & \cdots & 0 \\
                                     0 & \sigma_{A,2} & \ddots &\vdots \\
                                     \vdots & \ddots & \ddots & 0 \\
                                     0 & \cdots & 0 & \sigma_{A,k} \\
                                   \end{array}
                                 \right) U_A^*=V_A I_\sigma U_A^*.
\end{equation}
The vast majority of matrices in $\r_k$ satisfy
\begin{equation}\label{distinct}\sigma_{A,1}>\sigma_{A,2}>\ldots>\sigma_{A,k}>0,\end{equation}
and an arbitrarily small numerical perturbation will yield distinct
singular values. The subset of $\r_k$ satisfying (\ref{distinct})
will be denoted $\r_k^d$, where $d$ stands for distinct. If $\M$ is
a manifold and its closure $\overline{\M}$ is such that
$\overline{\M}\setminus\M$ is a union of manifolds of lower
dimension than $\M$, we will say that $\overline{\M}\setminus\M$ is
\textit{thin}. The proof of the following proposition can be found
in \cite{freqest_altproj}.
\begin{proposition}
$\r_k^d$ is a manifold of (real) dimension $2(2nk-k^2)$. Moreover,
$\r_k=\overline{\r_k^d}$ and $\r_k\setminus\r_k^d$ is thin.
\end{proposition}

The structure of the set $\r_k$ around the exceptional points in
$\r_k\setminus\r_k^d$ can be rather complicated, and although the
sequence generated by the alternating projection scheme could
theoretically approach such a point, we have not done any analysis
of convergence properties in this setting. This seems to be a hard
problem. The rate of convergence is however very poor around such
points (compare Figure \ref{fig1} with Figure \ref{fig3}).  The first
contains 12 iterations and the second 100.

We note that by the Eckart-Young theorem \cite{horn_johnsson}, it is
easy to project any matrix $B\in\m_{n,n}$ onto its closest point in
$\r_k$ while using the Frobenius norm $$\|B\|_{2}=\sqrt{\sum_{i,j}
|b_{i,j}|^2}.$$ Specifically, the theorem says that if
$B=V_BI_\sigma U_B$ is a singular value decomposition of $B$, then
the closest point in $\r_k$ is given by replacing $I_\sigma$ with
$I_\tau$ where $\tau=(\sigma_1,\ldots,\sigma_k,0,\ldots,0)$. The
closest point is thus unique as long as the singular values are
distinct, which is always the case when working with ``real
numerical'' data, so we will for simplicity treat the projection
onto $\r_k$ as a well defined map which we denote by $\pi_{\r_k}$.
More stringently one could work with ``point to set''-maps, as in
\cite{cadzow} and \cite{zangwill}.

The last manifold that needs to be discussed is
$$\H_k=\r_k\cap\H,$$ i.e., the set of Hankel matrices with rank $\leq
k$. It is easily seen that, given any $\alpha\in\C$, the matrix
\begin{equation}\label{hankelrank} H(\alpha)=\left(
                                            \begin{array}{ccccc}
                                              1 & \alpha & \alpha^2 & \cdots & \alpha^{n-1} \\
                                              \alpha & \alpha^2 & \iddots & \alpha^{n-1} & \alpha^n \\
                                              \alpha^2 & \iddots & \iddots & \iddots & \vdots \\
                                              \vdots & \alpha^{n-1} & \iddots & \iddots & \alpha^{2n-3} \\
                                              \alpha^{n-1} & \alpha^n & \cdots & \alpha^{2n-3} & \alpha^{2n-2} \\
                                            \end{array}
                                          \right)
\end{equation}
defines a rank 1 Hankel matrix, and thus
\begin{equation}\label{H}\mathfrak{H}(c,\alpha)=\sum_{j=1}^k c_jH(\alpha_j)\end{equation} has rank $k$. This is the typical
case, as we will show below. Let us denote the image of
$\mathfrak{H}$ intersected with $\r_k^d$ by $\H_k^n$, where $n$
stands for ``nice''. Again, the analysis is complicated by some
exceptional points. For example, $\frac{d^j}{d\alpha^j}H(\alpha)$ is
a Hankel matrix of rank $j$ which is not covered by $\mathfrak{H}$.
Moreover, rank $k$ Hankel matrices containing summands of this type
are usually inside $\r_k^d$, but despite that, some investigations
show that the manifold structure of $\H_k$ collapses around such
points. The sup-script $``n''(=nice)$ is thus really more restrictive
than $``d''(=distinct)$. Nevertheless, $\H_k\setminus\H_k^n$ is thin
and never encountered in practice, so we omit a study of such cases.
The proof of the following result can be found in \cite{freqest_altproj}.
\begin{proposition}\label{p5}
\begin{itemize}\item[]
\item[] $\H$ is a $2(2n-1)$-dimensional linear subspace.
\item[] $\H_k^n$ is a $4k$-dimensional manifold which is dense in
$\H_k$, its complement is thin.\item[] The map $\pi_{\r_k}$ is well
defined on all points in $\m_{n,n}$ except a for thin subset.
\end{itemize}
\end{proposition}
We now explain why one would like to do alternating projections
between $\H$ and $\r_k$. Let $\pi_{\H}$ be the orthogonal projection
onto $\H$, and let $\pi(B)$ denote the closest point to $B$ in
$\H\cap\r_k$. Say we have a sampled function on an interval, e.g.,
the signal in Figure \ref{fig_signal}. Denote the corresponding sequence
by $f$ and note that
\begin{equation}\label{eq53} \sum_{j=1}^{2n-1}|f_j|^2 (n-|j-n|)=\|H(f)\|_2^2.\end{equation} We denote the weight sequence in (\ref{eq53}) by
$w$ and the corresponding norm on $\C^{2n-1}$ by $\|f\|_w^2$. By
Proposition \ref{p5}, the function $g$ defining $\pi(B)=H(g)$ is
almost surely in the range of $\mathfrak{H}$, i.e. of the form
$$g=\sum_{j=1}^k c_j(\alpha_j^l)_{l=0}^{2n-2}.$$
Such functions are precisely what one encounters when sampling sums of $k$
exponential functions, and hence $g$ is the closest such function to $f$ in the norm given by (\ref{eq53}). Approximating a given function
with sums of a predetermined number of exponentials is a problem
with a large number of applications. There is, however, no computationally efficient method for computing $g$. In contrast, the projections $\pi_{\r_k}$ and $\pi_{\H}$ are relatively easy to compute, and hence, by the results of this paper, we can get a fairly good approximation of $g$ by finding the limit of the sequence of alternating projections, starting at $H(f)$.

For the sake of efficient computations, care has to be taken
concerning the implementation of $\pi_{\r_k}$ and $\pi_{\H}$.
Constructing fast algorithms for this purpose is one of the topics
of a companion paper \cite{freqest_altproj}, where we also discuss other
aspects of this specific application of alternating projections as
well as give proofs of the above claims. In the present paper, we
will continue discussing this application in Section \ref{numexp}.

\section{Preliminaries}\label{s1}
Let $\K$ be a Hilbert space of dimension $n\in\N$. Given $A\in\K$ and
$r>0$ we write $\B(A,r)$ or $\B_{\K}(A,r)$ for the open ball
centered at $A$ with radius $r$. Since $\K$ is finite-dimensional it
has a unique Euclidean topology. Any subset $\M$ of $\K$ will be
given the induced topology from $\K$.

\begin{definition}\label{def man}
We say that $\M\subset\K$ is locally an $m$-dimensional
$C^p$-manifold around $A\in\M$ if there exists $r_1>0$ and a
$C^p$-map $\phi:\B_{\R^m}(0,r_1)\rightarrow \K$ with the following
properties:
\begin{itemize}
\item $d\phi(x)$ is injective for all $x\in \B_{\R^m}(0,r_1)$,
\item $\phi(0)=A,$
\item $\phi$ is a homeomorphism onto an open neighborhood of $A$ in
$\M$.
\end{itemize}
\end{definition}
This is in line with the standard definition of $C^p$-manifolds, see
Theorem 2.1.2 \cite{berger_gostiaux} for a number of equivalent
definitions. As a consequence of the homeomorphism condition, note
that there exists an $s_1>0$ such that
\begin{equation}\label{eq77}\M\cap \B_{\K}(A,s_1)=\mathsf{Im} \phi\cap
\B_{\K}(A,s_1),\end{equation} where $\mathsf{Im} \phi$ denotes the
image of $\phi$. Given $B\in\mathsf{Im} \phi$, there exists a unique
$x\in\B(0,r_1)$ such that $B=\phi(x)$. We will without further comment
denote this $x$ by $x_B$. All the manifolds considered in this paper
are at least $C^1$, and hence we have
\begin{equation}\label{eq65}\phi(x)=C+d\phi(x_C)(x-x_C)+\textsl{o}(x-x_C)\end{equation} where
$\textsl{o}$ stands for ``little ordo''.\footnote{i.e. it stands for a
function with the property that $\textsl{o}(x)/\|x\|$ extends by
continuity to 0 and takes the value 0 there.} We define the tangent
space $T_\M(B)$ by $T_{\M}(B)=\Ran d\phi(x_B)$. It is a standard
fact from differential geometry that this definition is independent
of $\phi$. Moreover, we set
$$\tilde{T}_\M(B)=B+T_\M(B),$$
i.e., $\tilde{T}_\M(B)$ is the affine linear manifold which is
tangent to $\M$ at $B$. Throughout this section, $\M$ will be a
locally $C^p$-manifold at $A$, where $p\geq 1$ and we associate with
it $r_1$ and $s_1$ as in Definition \ref{def man} and (\ref{eq77}).
Moreover, there will follow a row of decreasing numbers $r_2,~r_3$
and $s_2, ~s_3$ etc., related to the above $A$, and the reader has to
bear in mind where these numbers were defined. The following
proposition basically says that the affine tangent-spaces are close
to $\M$ locally.

\begin{proposition}\label{p1}
Let $\M$ be a locally $C^1$-manifold at $A$. For each $\epsilon_2>0$
there exists $s_2,$ $0<s_2<s_1$, such that for all
$C\in\B(A,s_2)\cap\M$ we have
\begin{itemize}\item[($i$)]$ \dist(B,\tilde{T}_\M(C))\leq \epsilon_2\|B-C\|, \quad B\in
\B(A,s_2)\cap\M.$
\item[($ii$)]$\dist(B,\M)\leq \epsilon_2\|B-C\|, \quad $~$\quad B\in
\B(A,s_2)\cap T_\M(C).$
\end{itemize}
\end{proposition}
\begin{proof}

Given $r<r_1$, we first show that ${\|\phi(x)-\phi(y)\|}/{\|x-y\|}$
is uniformly bounded above and below for $x,y\in\B_{\R^m}(0,r)$.  By
(\ref{eq65}) and the mean value theorem we have
\begin{align*}&
\|\phi(y)-\phi(x)\|=\|d\phi(z)(y-x)\|
\end{align*}
for some $z$ on the line between $x$ and $y$. Now, $d\phi(z)$
depends continuously on $z$ and its singular values
$\sigma_1(d\phi(z)),\ldots,\sigma_m(d\phi(z))$ depend continuously
on the matrix entries \cite[p191]{Franklin}, hence
\begin{equation}\label{eq56}
\inf_{z\in\B_{\R^m}(0,r)} \{\sigma_n(d\phi(z))\}\|y-x\|\leq
\|\phi(y)-\phi(x)\|\leq \sup_{z\in\B_{\R^m}(0,r)}
\{\sigma_1(d\phi(z))\}\|y-x\|.
\end{equation}
By Definition \ref{def man}, $\sigma_m(d\phi(x))$ is never zero and
$cl(\B_{\R^m}(0,r))$ is compact, so both the $\inf$ and $\sup$
amount to finite positive numbers, as desired. ($cl$ denotes the
closure).

We now prove $(i)$; by (\ref{eq65}) and the mean value theorem we
have
\begin{align*}&\dist(B,\tilde{T}_\M(C))\leq
\|B-(C+d\phi(x_C)(x_B-x_C))\|=\|\phi(x_B)-\phi(x_C)-d\phi(x_C)(x_B-x_C))\|\leq\\&\leq
{\sup}\Big\{\|d\phi(y)-d\phi(x_C)\|:{y\in\B_{\R^m}(0,r)\text{ such
that }\|y-x_C\|\leq\|x_B-x_C\| }\Big\}\|x_B-x_C\|.
\end{align*}
Since $d\phi$ is continuous on the compact set $cl(\B_{\R^m}(0,r))$,
it is also equicontinuous. It follows that for each $\epsilon>0$ we
can pick a $\delta>0$ such that
\begin{equation*}{\sup}\big\{\|d\phi(y)-d\phi(x)\|:{x,y\in\B_{\R^m}(0,r)\text{
such that }\|y-x\|\leq\delta}\big\}<\epsilon.\end{equation*} Set
$r_2<\min(\delta/2,r)$ and let $s_2$, $0<s_2<s_1$, be such that
$B\in B_{\K}(A,s_2)\cap \M$ implies $\|x_B\|<r_2$, which we can do
since $\phi$ is a homeomorphism with $\phi(0)=A$. For
$B,C\in\B_{\K}(A,s_2)$, we then have (by equation (\ref{eq56})) that
\begin{align*}&\dist(B,\tilde{T}_\M(C))\leq
\epsilon\|x_B-x_C\|\leq \epsilon k\|\phi(x_B)-\phi(x_C)\|=\epsilon
k\|B-C\|,
\end{align*}
where $k=(\inf_{z\in\B_{\R^m}(0,r)} \{\sigma_n(d\phi(z))\})^{-1}$,
(cf equation (\ref{eq56})). By letting $\epsilon=\epsilon_2/k$, the
desired statement follows.

The proof of $(ii)$ is similar. Given $C$ let $y$ be such that
$B=C+d\phi(x_C) y$ and note that $\phi(x_C+y)\in\M$ so
\begin{align*}&\dist(B,\M)\leq \|\phi(x_C+y)-(C+d\phi(x_C) y)\|\leq{\sup}\Big\{\|d\phi(x_C+ty)-d\phi(x_C)\|:t\in[0,1]\Big\}\|y\|.
\end{align*}
Moreover, if $s_2$ is such that $\|x_C\|\leq r$, it is easily seen
that
$$\|y\|\leq 2 s_2\Big/\sup_{z\in\B_{\R^m}(0,r)}
\{\sigma_1(d\phi(z))\}.$$ Thus, by picking $s_2$ small enough we can
ensure that for all $B,C\in\B_{\K}(A,s_2)$, all points of the form
$x_C+ty$ are bounded by some pregiven $\delta>0$. From here the
proof follows a similar path as $(i)$, we omit the remaining
details.
%
%
%
%
%
\end{proof}

The value of $\epsilon_2$ will be determined later, so for the
moment we will consider it as a constant and the corresponding value
of $s_2$ will be kept for future use. However, $r_2$ was an internal
variable in the above proof, and to avoid too many subindices, we
let $r_2$ take a new value in the proof of the next proposition, (in
contrast to the value of $r_1$ which was given in Definition
\ref{def man} and connected to the fixed point $A$). The next
proposition shows that projection on $\M$ is a \textit{locally} well
defined operation.

\begin{proposition}\label{p2}
Let $\M$ be a locally $C^p$-manifold at $A$ with $p\geq 1$. Then
there exists $s_3>0$ and a $C^{p-1}$ map
$$\pi:\B_{\K}(A,s_3)\rightarrow \M$$ such that for all $B\in
\B_{\K}(A,s_3)$ there exists a unique closest point in $\M$ which is
given by $\pi(B)$. Moreover, $C\in\M\cap\B_{\K}(A,s_3)$ equals
$\pi(B)$ if and only if $B-C\perp T_\M(C)$.
\end{proposition}
\begin{proof}
Recall that $n$ is the dimension of $\K$ and $m$ the dimension of
$\M$ at $A$. By standard differential geometry there exists an
$r_2<r_1$ and $C^{p-1}$-functions
$f_1,\ldots,f_{n-m}:\B_{\R^m}(0,r_2)$ with the property that
$$\big(T_\M(\phi(x))\big)^\perp=\Span \{
f_1(x),\ldots,f_{n-m}(x)\}$$ for all $x\in\B_{\R^m}(0,r_2)$, (see
e.g. Theorem 2.7.7 in \cite{berger_gostiaux}). Define
$\sigma:\B_{\R^m}(0,r_2)\times \R^{n-m}\rightarrow\K$ via
$$\sigma(x,y)=\phi(x)+\sum_{i=1}^{n-m}y_if_i(x).$$
Consider the set $\mathcal{S}\subset\K$ of points whose multiplicity
under $\sigma$ is greater than 1, (i.e. all points hit more than
once by $\sigma$). By the inverse function theorem, the set
$\sigma^{-1}(\mathcal{S})$ can not have 0 as an accumulation-point,
for it says that there exists an $r<r_2$ such that $\sigma$
restricted to ${\B_{\R^n}(0,r)}$ is a diffeomorphism onto its image.
Pick $r_3<\min\big(r,\dist(\sigma^{-1}(\mathcal{S}),0)\big)$ and
pick $s_3$ such that
\begin{equation}\label{eq1}\M\cap
\B_{\K}(A,2s_3)=\phi(\B_{\R^m}(0,r_3))\cap
\B_{\K}(A,2s_3),\end{equation} and
\begin{equation}\label{eq3}\B_\K(A,s_3)\subset\sigma({\B_{\R^n}(0,r_3)}).\end{equation}
Note that, given $B\in\B_\K(A,s_3)$ there exists a unique
$(x_B,y_B)$
such that
$B=\sigma((x_B,y_B))$ and moreover $\|(x_B,y_B)\|\leq r_3$ by (\ref{eq3}). 
We define $$\pi(B)=\phi(x_B).$$ To see that $\pi$ is a
$C^{p-1}$-map, let $\theta:\R^m\times\R^{n-m}\rightarrow\R^m$ be
given by $\theta((x,y))=x$ and note that
$$\pi=\phi\circ\theta\circ(\sigma|_{\B_{\R^n}(0,r_3)})^{-1}.$$
We now show that $\pi(B)$ have the desired properties. By the
construction,
$$\pi(B)-B\in \Span \{ f_1(x_B),\ldots,f_{n-m}(x_B)\}\perp
T_\M(\phi(x_B))=T_\M(\pi(B)).$$ Now suppose $C\in\M$ is a closest
point to $B$. Since $\|A-B\|<s_3$ we clearly must have
$\|C-A\|<2s_3$ so by (\ref{eq1}) there exists a
$x_C\in\B_{\R^m}(0,r_3)$ with $\phi(x_C)=C$. Since $r_3<r_1$ we know
that $\M$ is completely determined by $\phi$ in the vicinity of $C$.
In particular, it makes sense to talk about $T_{\M}(C)$ and it is
easily seen that $B-C \perp T_{\M}(C)$, for by (\ref{eq65}) we have
\begin{align*}&\|\phi(x)-B\|^2=\|C+d\phi(x_C)(x-x_C)+\textsl{o}(x-x_C)-B\|=\\&=\|C-B\|^2+2\langle
C-B, d\phi(x_C)(x-x_C) \rangle+\textsl{o}(\|x-x_C\|)\end{align*} and
hence the scalar product needs to be zero for all $x$'s. Thus there
is a $y$ such that $B=\sigma((x_C,y))$. But since $(x_B,y_B)$ is the
unique point with this property, we deduce that $x_B=x_C$ and hence
$\pi(B)=\phi(x_B)=\phi(x_C)=C$. This establishes the first part of
the proposition. Now let $C$ be as in the second part of the
proposition. As above we have $C=\phi(x_C)$ with $\|x_C\|<r_3$ and
the orthogonality implies that there exists a $y$ with
$B=\sigma((x_C,y))$ and again this implies $C=\pi(B)$, as desired.
\end{proof}
\section{Non-tangentiality}\label{s2}
Suppose now that we are given closed sets $\M_1$ and $\M_2$ which
locally are manifolds around an intersection point
$A\in\M_1\cap\M_2$.

\begin{definition} An intersection point $A$ will be called regular if there are numbers
$m_1,$ $m_2$, $m$ and $p\geq 1$ such that
\begin{itemize}
\item $\M_j$ is locally an $m_j$-dimensional $C^p$-manifold at $A$, $j=1,2.$
\item $\M_1\cap\M_2$ is locally an $m$-dimensional $C^p$-manifold at $A$.
\end{itemize}
\end{definition}
Note that the set of regular points is clearly a relatively open set
in $\M_1\cap\M_2$. Next, we introduce angles. For more information
on angles, we refer to \cite{lewis_malick}.
\begin{definition}\label{def angle} For any regular point $A$, we define the angle $\alpha(A)$ of $\M_1$ and $\M_2$ at $A$ to be the $\cos^{-1}$ of the number
$$\sigma(A)=\lim_{r \rightarrow 0} \sup \left\{\frac{\langle B_1-A,B_2-A\rangle}{\|B_1-A\|\|B_2-A\|}:B_j\in\M_j,\|B_j-A\|<r \text{ and } B_j-A\perp T_{\M_1\cap\M_2}(A)\right\}.$$
\end{definition}
Given a linear subspace $\M$ we let $P_{\M}$ denote the orthogonal
projection onto $\M$. It is easily verified that
\begin{equation}\label{eq777}\begin{aligned}&\sigma(A)=\sup \left\{\frac{\langle
B_1-A,B_2-A\rangle}{\|B_1-A\|\|B_2-A\|}:B_j\in T_{\M_j}(A)\ominus
T_{\M_1\cap\M_2}(A)\right\}=\\&=\|P_{T_{\M_1}(A)\ominus
T_{\M_1\cap\M_2}(A)}P_{T_{\M_2}(A)\ominus
T_{\M_1\cap\M_2}(A)}\|=\|P_{T_{\M_1}(A)}P_{T_{\M_2}(A)}-P_{T_{\M_1\cap\M_2}(A)}\|.\end{aligned}\end{equation}
If $\M_1$ and $\M_2$ are hyperplanes through the origin, then it is
clear that the above definition coincides with the classical
definition;\begin{equation}\label{eq54} \cos
\alpha_{clas}(A)=\sigma_{clas}(A)=\|P_{\M_1}P_{\M_2}-P_{\M_1\cap\M_2}\|.
\end{equation}
However, it is important to note that
$\alpha_{clas}(A,T_{\M_1(A)},T_{\M_2(A)})$ and $\alpha(A)$ are not
necessarily the same. For example, take
$\M_1=\{(x_1,x_2,0,0,x_3):x\in\R^3\}$ and
$\M_2=\{(x_1,x_2,x_2^2,x_3,x_3):x\in\R^3\}$. Then
$\M_1\cap\M_2=\{(x,0,0,0,0):x\in\R\}$ and $0$ is a regular point.
Moreover $\alpha_{clas}(0,T_{\M_1(0)},T_{\M_2(0)})=\pi/4$ whereas
$\alpha(0)=0$. What goes wrong above is clearly that the two
surfaces are tangential to each other in the direction
$(0,0,1,0,0)$, and it is intuitively clear that when this is not the
case, the two concepts should coincide. To avoid such obstacles, we
therefore introduce:

\begin{definition}\label{def st}
$\M_1$ and $\M_2$ are said to be non-tangential at $A$ if $A$ is
regular and they have a positive angle at $A$, i.e. if
$\sigma(A)<1$. We will often simply say that $A$ is non-tangential.
\end{definition}
Note that the angle is always a number between 0 and $\pi/2$, so in
an intuitive sense only exceptional points do not satisfy
non-tangentiality. We have no intention of formalizing this
statement, but point out already that for the example considered in
Section \ref{I}, this is indeed the case, which will be proven in
\cite{freqest_altproj}.

\begin{theorem}\label{t3}
With the notation as in Section \ref{I}, we have that the set of
tangential points between $\r_k$ and $\H$ is thin in $\H_k$.
\end{theorem}

In the general setting, pathological examples do exist. For example,
take $\M_1=\{(x_1,x_2,0):x\in\R^2\}$ and
$\M_2=\{(x_1,x_2,x_1^2):x\in\R^2\}$. However, when $\M_1$ and $\M_2$
are defined by polynomials, one can use similar methods, as in
the proof of Theorem \ref{t3}, to show that under mild conditions, if
non-tangentiality holds at one point, then it holds everywhere
except for a thin set. We will not pursue this.

By (\ref{eq777}) it is easy to see that $A$ is non-tangential if and
only if
\begin{equation}\label{eq angle}\Big(T_{\M_1}(A)\ominus
T_{\M_1\cap\M_2}(A)\Big)\cap\Big(T_{\M_2}(A)\ominus
T_{\M_1\cap\M_2}(A)\Big)= \{0\}\end{equation} which in turn happens
if and only if
\begin{equation}\label{eq angle1} T_{\M_1}(A)\cap T_{\M_2}(A)=
T_{\M_1\cap\M_2}(A).\end{equation} This latter condition is usually
referred to as \textit{clean intersection} in microlocal analysis.
We have chosen the terminology non-tangential since it is more
intuitive. We now show that non-tangentiality is a weaker concept
than transversality, defined in (\ref{transversal}).

\begin{proposition}
Let $\M_1$ and $\M_2$ be two $C^1$-manifolds and $A\in\M_1\cap\M_2$
a transversal point. A transversal point $A$ is also non-tangential.
\end{proposition}
\begin{proof}
Applying the implicit function theorem to $\phi_1-\phi_2$ easily
yields that $\M_1\cap\M_2$ is a $C^1$-manifold of dimension
$m_1+m_2-n$ at $A$. Thus $T_{\M_1}(A)\ominus T_{\M_1\cap\M_2}(A)$
has dimension $n-m_2$ and $T_{\M_2}(A)\ominus T_{\M_1\cap\M_2}(A)$
has dimension $n-m_1$. Thus
$$\dim(T_{\M_1}(A)\ominus
T_{\M_1\cap\M_2}(A))+\dim(T_{\M_2}(A)\ominus
T_{\M_1\cap\M_2}(A))+\dim(T_{\M_1\cap\M_2}(A))= n.$$ But by
transversality the sum of the above subspaces equals $\K$, which can
only happen if (\ref{eq angle}) is satisfied.
\end{proof}

We will denote $\M_1\cap\M_2$ by $\M$, and the objects from Section
\ref{s1} associated to $\M_1$, $\M_2$ and $\M$, e.g. $\phi$, by
$\phi_1$, $\phi_2$ and $\phi$ respectively. We thus omit subindex
when dealing with $\M_1\cap\M_2$. We now prove that for
non-tangential points, the angle as defined here and the classical
angle of the respective tangent spaces coincide.

\begin{proposition}\label{prop45}
If $A$ is non-tangential, then
$$\alpha_{clas}(A,T_{\M_1(A)},T_{\M_2(A)})=\alpha(A).$$
\end{proposition}
\begin{proof}
Using (\ref{eq777}) it is easy to see that
\begin{align*}\sigma(A)=\|P_{T_{\M_1}}P_{T_{\M_2}}-P_{T_{\M_1\cap\M_2}}\|=\|P_{T_{\M_1}}P_{T_{\M_2}}-P_{T_{\M_1}\cap
T_{\M_2}}\|=\sigma_{clas}(A,T_{\M_1},{T_{\M_2}}),\end{align*} where
we used $(\ref{eq angle1})$ in the crucial step.
\end{proof}

\begin{proposition}\label{p3}
The function $\sigma$ in Definition \ref{def angle} is $C^{p-1}$ (on
the set of regular points in $\M_1\cap\M_2$). In particular,
non-tangentiality is a local property, i.e. if $A$ is
non-tangential, then the same holds for all $B\in\M_1\cap\M_2$ in a
neighborhood of $A$.
\end{proposition}
\begin{proof}
Let $A\in\M_1\cap\M_2$ be a regular point and let $\phi$ be the usual chart. We need to show that $\sigma\circ\phi$ is $C^{p-1}$ around 0. 
By standard differential geometry (see e.g. \cite{berger_gostiaux})
there exists $C^{p-1}$-functions $f_1,\ldots,f_{n-m}$ defined on the
domain of $\phi$ such that
$$\big(T_{\M_1\cap\M_2}(\phi(x))\big)^\perp=\Span\{f_1(x),\ldots,f_{n-m}(x)\}.$$
It is easy to see that (with $j=1$ or $j=2$)
\begin{equation}\label{eq4}T_{\M_j}(\phi(x))\ominus T_{\M_1\cap\M_2}(\phi(x))=\Span
\{P_{T_{\M_j}(\phi(x))} f_1(x), \ldots P_{T_{\M_j}(\phi(x))}
f_{n-m}(x)\}.\end{equation} Since the dimension of
$T_{\M_j}(\phi(x))$ is constant $m_j$, it is easy to see that the
functions on the right are $C^{p-1}$ as well. Moreover, as the
dimension of the space on the left in (\ref{eq4}) is constant
$m_j-m$, we can pick $m_j-m$ of the functions on the right and put
them as columns of an injective matrix $M_j(x)$ such that
$$T_{\M_j}(\phi(x))\ominus T_{\M_1\cap\M_2}(\phi(x))=\Ran M_j(x)$$
for $x$ in a neighborhood of 0. By $(\ref{eq777})$ we have for $x$
around 0 that
\begin{equation}\label{eq5}\sigma(\phi(x))=\|P_{\Ran M_1(x)}P_{\Ran
M_2(x)}\|.\end{equation} (Note that this is not true with the
classical definition when the two ranges have a non-trivial
intersection, but with the definition here it works.) Moreover, it
is easy to see that $$P_{\Ran M_j}=M_j(M_j^*M_j)^{-1}M_j^*,$$ where
$x$ has been omitted for readability. Combining this with
(\ref{eq5}) and the fact that singular values are
$C^\infty$-functions of the matrix entries, it is clear that
$\sigma$ is $C^{p-1}$ near $0$, as desired.
\end{proof}


\section{Properties of the projection operators}
Let $A$ $\M_1$, $\M_2$, $\phi_1$ etc. be as before, i.e., $C^{p}$-manifolds with a non-tangential intersection point $A$. We assume that $p\geq 1$ and continue to use the convention of denoting objects related to the
$C^p$-manifold $\M=\M_1\cap\M_2$ without subindex. In Propositions
\ref{p1} and \ref{p2} the quantities $s_2$ and $s_3$ appears. These are not necessarily the same and they depend on an auxiliary constant
$\epsilon_2$ originating from Proposition \ref{p1}. In this section we let $\epsilon_2$ be a fixed number (which we will determine later), and we let $s_4$ denote the minimum of all possible $s$'s from Section
\ref{s1} related to the 3 manifolds. The above will not be repeated
in the statements of the results below. Thus, we can apply any result
from Section \ref{s1} to either of the manifolds considered here.
Moreover, letting $j$ denote either $1,2$ or nothing, we let $r_4>0$
be such that
\begin{equation*}\M_j\cap \B_{\K}(A,s_4)=\phi_j(\B(0,r_4))\cap
\B_{\K}(A,s_4),\end{equation*} and such that the results of Section
\ref{s1} applies to each $\phi_j(x)$ with $\|x\|<r_4$. We also
assume that $\epsilon_2<1$.

\begin{figure}[tbh!]\label{fig5}
\unitlength=1in
\begin{center}
\psfrag{M1M2}[][l]{$\M_1 \cap \M_2$} \psfrag{piB}[][l]{$\pi(B)$}
\psfrag{rho1}[][l]{$\rho_1(B)$} \psfrag{pi1}[][l]{$\pi_1(B)$}
\psfrag{Bpos}[][l]{$B$} \psfrag{E}[][l]{$\e$} \psfrag{F}[][l]{$\f$}
\psfrag{D}[][l]{$\d$} \psfrag{M1}[][l]{$\M_1$}
\psfrag{r2}[][l]{$s_4$} \psfrag{F}[][l]{$\f$}
\includegraphics[width=3.5in]{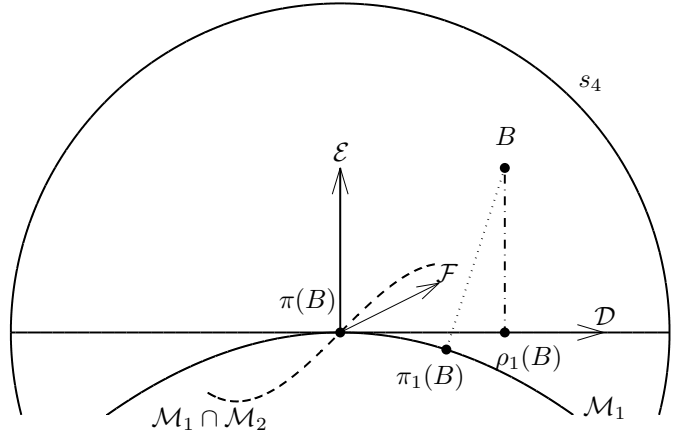}
\end{center}
\caption{Illustration of the difference between $\rho_1$ and
$\pi_1$. $\d$ and $\e$ appear in the proof of Proposition \ref{p4}.}
\end{figure}

Given an affine linear manifold $\n\subset\K$, we denote by $P_\n$
the orthogonal projection onto $\n$. We introduce maps
$\rho_j:\B(A,s_4)\rightarrow\K$, ( $j=1$ or $j=2$), via
$$\rho_j(B)=P_{\tilde{T}_{\M_j}(\pi(B))}(B).$$
Thus, $\rho_j$ resemble $\pi_j$ but is slightly different. $\pi_j$
projects onto $\M_j$ whereas $\rho_j$ projects onto the tangent
plane of $\M_j$ taken at the closest point to $B$ in $\M_1\cap\M_2$,
i.e. $\pi(B)$, (see Fig \ref{fig5}). A proper estimate of the
difference is given in Proposition \ref{p4}.

\begin{lemma}\label{l1}
The operators $\rho_1$ and $\rho_2$ are $C^{p-1}$-maps in $\B_{\K}(A,s_4)$.
Moreover, we can select a number $s_5<s_4$ such that the image of
$\B(A,s_5)$ under $\rho_1,\rho_2,\pi,\pi_1,\pi_2$, as well as any
composition of two of those maps, is contained in $\B(A,s_4)$.
\end{lemma}
\begin{proof}
The second part is an immediate consequence of the continuity of the
maps. Let $j$ denote $1$ or $2$ and set $M_j(B)=d\phi_j(\pi(B))$.
Note that $M_j$ is a $C^{p-1}$-map in $\B_{\K}(A,s_4)$ by
Proposition \ref{p2} and the choice of $s_4$. Moreover,
\begin{align*}&\rho_j(B)=P_{\tilde{T}_{\M_j}(\pi(B))}(B)=\pi(B)+P_{{T}_{\M_j}(\pi(B))}(B-\pi(B))=\\&=\pi(B)+M_j(B)\big(M_j^*(B)M_j(B)\big)^{-1}M_j^*(B)\big(B-\pi(B)\big),
\end{align*}
from which the result follows.
\end{proof}


\begin{proposition}\label{p4}
Given any $B\in\B(A,s_5)$ and $j=1$ or $j=2$, we have
$$\|\pi_j(B)-\rho_j(B)\|< 5\sqrt{\epsilon_2}\|B-\pi(B)\|.$$
\end{proposition}
\begin{proof}
By Lemma \ref{l1} we have that Proposition \ref{p1} applies to the
point $C=\pi(B)$. It is no restriction to assume that $\pi(B)=0$,
which we now do. Denote $\d=T_{\M_j}(0)$, $\e=\Span\{B-\rho_j(B)\}$
and $\f=\K\ominus(\d\oplus\e)$, (see Fig. \ref{fig5}). Let $D_B$ and
$E_B$ be elements of $\d$ and $\e$ such that $B=D_B+E_B$, and note
that
$$\rho_j(B)=D_B.$$ We thus have to show that
\begin{equation}\label{hy0}\|\pi_j(B)-D_B\|<
5\sqrt{\epsilon_2}\|B\|.\end{equation} First note that by
Proposition \ref{p1} ($ii$) there exists a point in
$\M_1\cap\M_2\cap\B(D_B,\epsilon_2\|D_B\|)$. Thus
$$\|B-\pi_j(B)\|\leq \|B-D_B\|+\epsilon_2\|D_B\|=\|E_B\|+\epsilon_2\|D_B\|$$
and hence $\pi_j(B)$ is a member of the set
\begin{equation}\label{hy1}\Big\{(D,E,F):\|D-D_B\|^2+\|E-E_B\|^2+\|F\|^2<(\|E_B\|+\epsilon_2\|D_B\|)^2\Big\}.\end{equation}
However, by Proposition \ref{p1} ($i$) we have that $\pi_j(B)$ also
is a member of
\begin{equation}\label{hy2}\Big\{(D,E,F):\|E\|^2+\|F\|^2<\epsilon_2^2\|D\|^2\Big\}.\end{equation}
The left hand side of (\ref{hy0}) is thus dominated by the supremum
of the function $$f(D,E,F)=\|D-D_B\|^2+\|E\|^2+\|F\|^2,$$ subject to
the conditions in (\ref{hy1}) and (\ref{hy2}). Either by geometrical
considerations or the method of Lagrange multipliers, it is not hard
to deduce that this supremum is attained for $F=0$ and $D,E$ of the
form $D=dD_B$ and $E=eE_B$ where $d,e\in\R$. We now have a
two-dimensional problem of circles and cones, and in the remainder
of the proof we treat $D$ and $E$ as elements of $\R^2$. Given $D+E$
in the intersection of (\ref{hy1}) and (\ref{hy2}), it is easily
seen (see Fig. \ref{fig6}) that
$$\|E\|\leq \epsilon_2(\|D_B\|+\|E_B\|+\epsilon_2\|D_B\|),$$ since
this is the height of the cone at the outer edge of the circle. The
problem becomes simpler if we replace the cone in (\ref{hy2}) by the
following strip:
\begin{equation}\label{hy3}\Big\{E:\|E\|< \epsilon_2((1+\epsilon_2)\|D_B\|+\|E_B\|)\Big\}.\end{equation}
\begin{figure}[tbh!]\label{fig6}
\unitlength=1in
\begin{center}
\psfrag{alpha}[][l]{$\alpha$} \psfrag{EB}[][l]{$E_B$}
\psfrag{DB}[][l]{$D_B$} \psfrag{M1}[][l]{$\M_1$}
\psfrag{EBn}[][l]{${\tiny{\|E_B\|+\epsilon_2\|D_B\| }}$}
\psfrag{Eexp}[][l]{${\tiny{\epsilon_2((1+\epsilon_2)\|D_B\|+\|E_B\|)}}$}
\includegraphics[width=4.5in]{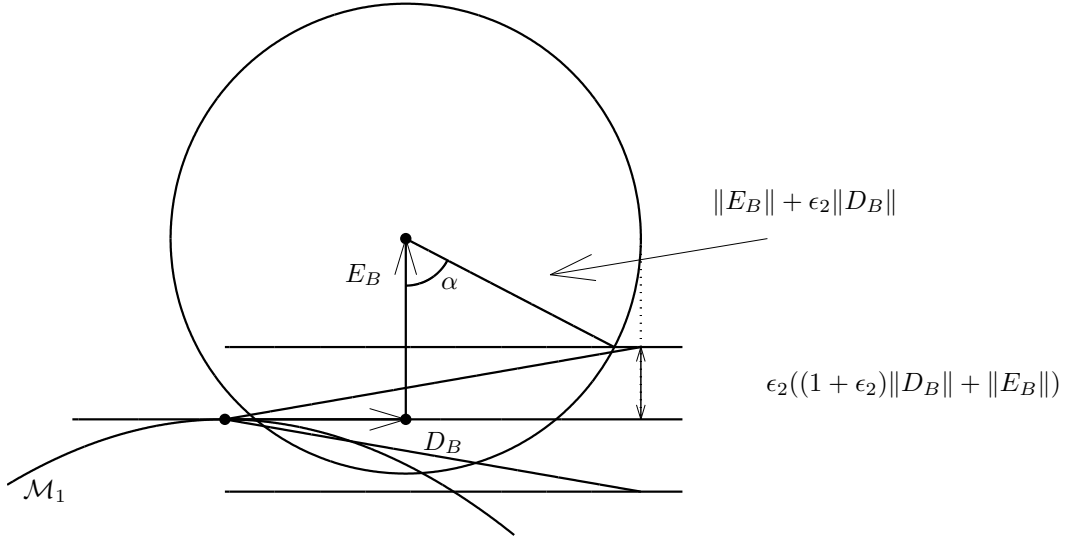}
\end{center}
\caption{Proof illustration.}
\end{figure}
From Figure \ref{fig6}, and some freshman
formulas, it is readily verified that the sought supremum is dominated by $\alpha(\|E_B\|+\epsilon_2\|D_B\|),$ where $\alpha$ is the angle
given by
$$\alpha=\cos^{-1}\left(\frac{\|E_B\|-\epsilon_2((1+\epsilon_2)\|D_B\|+\|E_B\|)}{(\|E_B\|+\epsilon_2\|D_B\|)}\right).$$
Since we have assumed that $\epsilon_2<1$, the proposition thus
follows if we establish that
$$\cos^{-1}\left(\frac{\|E_B\|-\epsilon_2(2\|D_B\|+\|E_B\|)}{(\|E_B\|+\epsilon_2\|D_B\|)}\right)\frac{\|E_B\|+\epsilon_2\|D_B\|}{\sqrt{\|E_B\|^2+\|D_B\|^2}}<5\sqrt{\epsilon_2}$$
By dividing with $\|E_B\|$ at suitable places, one sees that this is
equivalent to
$$\sup_{t>0}\left\{\cos^{-1}\left(\frac{1-2\epsilon_2t-\epsilon_2}{1+\epsilon_2 t}\right)\frac{1+\epsilon_2 t}{\sqrt{1+t^2}}\right\}<5\sqrt{\epsilon_2}$$
Using $\cos^{-1}(1-x)\leq \pi\sqrt{x/2}$ and $\epsilon_2<1$ we
obtain
\begin{align*}
&\cos^{-1}\left(\frac{1-2\epsilon_2t-\epsilon_2}{1+\epsilon_2
t}\right)\frac{1+\epsilon_2
t}{\sqrt{1+t^2}}=\cos^{-1}\left(1-\frac{3\epsilon_2t+\epsilon_2}{1+\epsilon_2
t}\right)\frac{1+\epsilon_2 t}{\sqrt{1+t^2}}\leq
\\& \leq \frac{\pi}{\sqrt{2}}\sqrt{\frac{3\epsilon_2t+\epsilon_2}{1+\epsilon_2
t}}\frac{1+\epsilon_2 t}{\sqrt{1+t^2}}\leq\sqrt{\epsilon_2}
\frac{\pi}{\sqrt{2}}\frac{\sqrt{{(3t+1)}{(1+t)}}}{\sqrt{1+t^2}}<
5\sqrt{\epsilon_2}
\end{align*}
since the function on the right clearly is bounded, and a few
calculations show that the maximum is attained at $t=(1+\sqrt{5})/2$
which gives $5\sqrt{\epsilon_2}$ as an upper bound.
\end{proof}


\begin{lemma}\label{l3}
Given $B\in\B(A,s_5)$ we have
$$\pi(B)=\pi(\rho_1(B))=\pi(\rho_2(\rho_1(B))).$$
\end{lemma}
\begin{proof}
If we prove the first equality the second follows by reversing the
roles of 1 and 2 and applying the first equality to $\rho_1(B)$. To
see the first equality, note that $B-\rho_1(B)\perp
T_{\M_1}(\pi(B))$ and obviously $B-\pi(B)\perp T_{\M_1\cap
M_2}(\pi(B))$ so
$$\rho_1(B)-\pi(B)=(\rho_1(B)-B)+(B-\pi(B))\perp T_{\M_1\cap
M_2}(\pi(B)),$$ which by Lemma \ref{l1} and Proposition \ref{p2}
implies $\pi(B)=\pi(\rho_1(B))$, as desired.
\end{proof}

\begin{lemma}\label{l2}
Let $c_5>1$ and $\epsilon_5>0$ be given. If $E,F\in\K$ satisfies
$\|E\|>c_5\|F\|$ and $\|E-F\|<\epsilon_5$, then
$$\|E\|<\epsilon_5\frac{c_5}{c_5-1}.$$
\end{lemma}
\begin{proof} If $\|E\|<\epsilon_5$ we are done, otherwise
$$c_5<\frac{\|E\|}{\|F\|}<\frac{\|E\|}{\|E\|-\epsilon_5}$$
which easily gives the desired estimate.
\end{proof}

The next result will be the main tool for proving convergence of the
alternating projections.
\begin{theorem}\label{t1}
Let $A\in\M_1\cap\M_2$ be a non-tangential point and assume that
$\M_1$ and $\M_2$ are $C^2$-manifolds. Then for each $c_6>\sigma(A)$
and each $\epsilon_6>0$ there exists a positive $s_6<s_5$ such that
for all $B\in\M_2\cap\B(A,s_6)$ we have
\begin{itemize}
\item[($i$)] $\|\pi_1(B)-\pi(B)\|<c_6\|B-\pi(B)\|$
\item[($ii$)] $\|\pi(\pi_1(B))-\pi(B)\|< \epsilon_6 \|B-\pi(B)\|$
\end{itemize}
Moreover the same holds true with the roles of $\M_1$ and $\M_2$
reversed.
\end{theorem}
\begin{proof}
Fix $c_1$ such that $\sigma(A)<c_1<c_6$ and pick an $s_7<s_5$ such
that
\begin{equation}\label{ghdf}\sup\{\sigma(C):C\in\M_1\cap\M_2\cap\B(A,s_7)\}<c_1,\end{equation} which we can
do since $\sigma$ is continuous by Proposition \ref{p3}. Let $c_2>1$
be such that $c_2 c_1<{c_6}.$ By Lemma \ref{l1} and Proposition
\ref{p2}, $\rho_1$ and $\pi$ are $C^1$-functions, and hence we can
pick $c_3>0$ such that
\begin{equation}\label{eq57}\|\rho_j(B)-\rho_j(B')\|\leq c_3\|B-B'\|\text{ and
}\|\pi(B)-\pi(B')\|\leq c_3\|B-B'\|\end{equation} for all
$B,B'\in\B(A,s_4)$, (recall that $s_4$ was chosen in the beginning
of this section). Finally, we fix $\epsilon_2$ such that
\begin{equation}\label{gf} 5\sqrt{\epsilon_2}(c_3+c_3^2)<\epsilon_6 \text{ and }
5\sqrt{\epsilon_2}(1+c_3)\frac{c_2}{c_2-1}<c_6\text{ and
}(1+5\sqrt{\epsilon_2})c_2c_1<c_6.\end{equation} This may seem like
a circle argument, because $c_1,c_2$ and $c_3$ depends on $c_6$
which depends on $\epsilon_2$ via $s_5$, and the $c_j$'s appears in
(\ref{gf}). However, this is easily circumvented by first choosing
$s_5$ with $\epsilon_2=1$, say, and pick values of $c_1$, $c_2$,
$c_3$. Then, once the real $\epsilon_2$ has been chosen via
(\ref{gf}) we can redefine $s_4,~s_5$ and $s_7$ accordingly without
violating (\ref{ghdf}) or (\ref{eq57}).

\begin{figure}[ht]\label{fig12}
\centering \psfrag{OOO}[c]{0} \psfrag{AAA}[c]{$A$}
\psfrag{BBp}[c]{$B'$} \psfrag{BB}[c]{$B$} \psfrag{DDp}[c]{$D'$}
\psfrag{DD}[c]{$D$} \psfrag{M1}[c]{$T_{\mathcal{M}_1}(0)$}
\psfrag{M2}[c]{$\mathcal{M}_2$} \psfrag{M1M2}[c]{$\mathcal{M}_1 \cap
\mathcal{M}_2$} \psfrag{TM1}[c]{$\mathcal{M}_1$}
\psfrag{TM2}[c]{$T_{\mathcal{M}_2}(0)$}
\includegraphics[width=.5\linewidth]{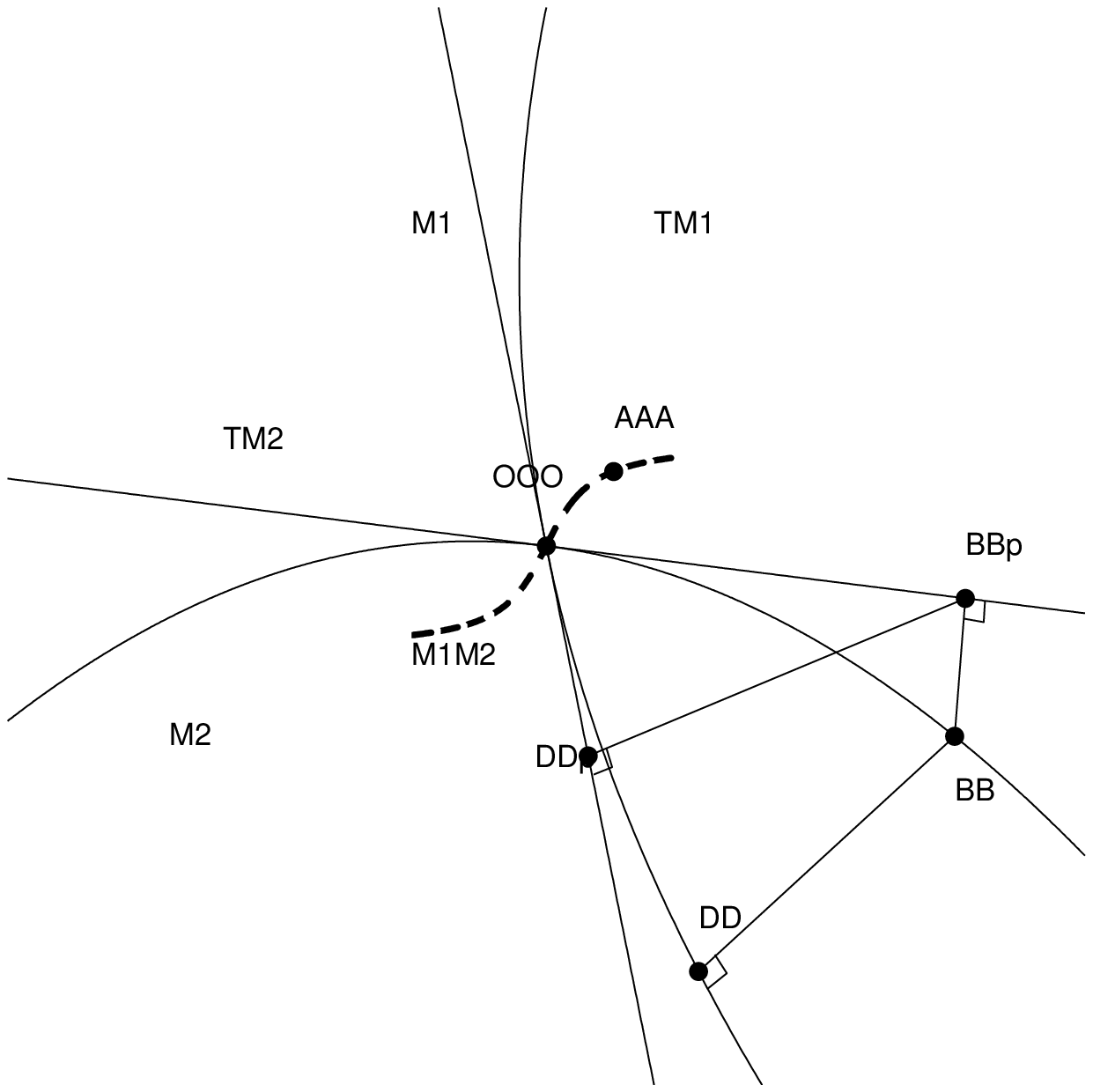}
\end{figure}

Now, pick an $s_6$ such that
$$\pi(\B(A,s_6))\subset\B(A,s_7),$$ and let $B\in\M_2\cap\B(A,s_6)$.
We begin with proving $(ii)$. Denote $C=\pi(B)$ and $D=\pi_1(B)$ and
note that $C\in\M_1\cap\M_2\cap\B(A,s_7)$ so $\sigma(C)<c_6$. There
is no restriction to assume that $C=0$, which we do from now on. We
thus need to show that $\|\pi(D)\|<\epsilon_6\|B\|$. Put
$B'=\rho_2(B)$ and $D'=\rho_1(B')$. (See Figure \ref{fig12} and
recall the $\pi(B')=\pi(B)=0$ by Lemma \ref{l3}). First, note that by
Proposition \ref{p4}
\begin{equation}\label{apa}\|B-B'\|=\|\pi_2(B)-\rho_2(B)\|<5\sqrt{\epsilon_2}\|B\|\end{equation}
and moreover by (\ref{eq57}) and Proposition \ref{p4} we have that
\begin{equation}\label{ok}\begin{aligned} &\|D-D'\|=\|\rho_1(B')-\pi_1(B)\|\leq
\|\rho_1(B')-\rho_1(B)\|+\|\rho_1(B)-\pi_1(B)\|\leq\\&\leq
c_3\|B'-B\|+5\sqrt{\epsilon_2}\|B\|<5\sqrt{\epsilon_2}(1+c_3)\|B\|
\end{aligned}\end{equation}
By Lemma \ref{l3} we have $0=\pi(B)=\pi(D')$ so part ($ii$) follows
by (\ref{eq57}), (\ref{gf}), (\ref{ok}) and the calculation
$$\|\pi(D)\|=\|\pi(D)-\pi(D')\|< c_3(5\sqrt{\epsilon_2}(1+c_3)\|B\|)<\epsilon_6 \|B\|.$$
We turn to part ($i$). Clearly $B'\in T_{\M_2}(0)$ and by Lemma
\ref{l3} we also have $D'\in T_{\M_1}(0).$ Thus
\begin{equation}\label{po}\frac{\|D'\|}{\|B'\|}\leq\sigma(0)<c_1\end{equation} whereas ($i$) amounts
to showing that $\|D\|/\|B\|<c_6$. Recall (\ref{ok}) and apply Lemma
\ref{l2} with $E=D$, $F=D'$,
$\epsilon_5=5\sqrt{\epsilon_2}(1+c_3)\|B\|$ and $c_5=c_2$. We see
that either
\begin{equation}\label{gj} \frac{\|D\|}{\|D'\|}\leq c_2
\end{equation}
or $\|D\|< 5\sqrt{\epsilon_2}(1+c_3)\frac{c_2}{c_2-1}\|B\|$, in
which case we are done since the constant is less than $c_6$ by
(\ref{gf}). We thus assume that (\ref{gj}) holds. Note that
\begin{equation}\label{po1} \frac{\|B'\|}{\|B\|}\leq 1+5\sqrt{\epsilon_2}
\end{equation}
by (\ref{apa}). Combining (\ref{po}), (\ref{gj}) and (\ref{po1}) we
get
$$\frac{\|D\|}{\|B\|}=\frac{\|D\|}{\|D'\|}\frac{\|B'\|}{\|B\|}\frac{\|D'\|}{\|B'\|}<c_2(1+5\sqrt{\epsilon_2})c_1<c_6$$
by (\ref{gf}).
\end{proof}

\section{Alternating projections}\label{ap}
We are finally ready for the main theorem. The third conclusion
below was not mentioned in the introduction. It basically says that
if the angle between $\M_1$ and $\M_2$ is not too close to 0, the
sequence of alternating projections $B_0,~B_1,~B_2,\ldots$ will
converge within machine precision within a fairly low number of
iterations. In the terminology of \cite{check}, $(B_k)_{k=0}^\infty$
converges ``R-linearly'' with rate less than $c$. For example, if
$\sigma(A)=1/2$, then we will hit single precision ($\approx
10^{-7.5}$) at $B_{24}$, since single precision has 24 bits of
significant precision.

\begin{theorem}\label{main}
Let $\M_1\cap\M_2$ be locally non-tangential $C^2$-manifolds around
$A\in\M_1\cap\M_2$, and let $\epsilon>0$ and $1>c>\sigma(A)$ be
given. Then there exists an $s>0$ such that the sequence of
alternating projections
\begin{equation*}B_0=\pi_1(B),~B_1=\pi_2(\pi_1(B)),~B_2=\pi_1(\pi_2(\pi_1(B))),~B_3=\pi_2(\pi_1(\pi_2(\pi_1(B)))),\ldots\end{equation*}~
\begin{itemize}
\item[($i$)] converges to a point $B_\infty\in\M_1\cap\M_2$\\
\item[($ii$)] $\|B_\infty-\pi(B)\|<\epsilon\|B-\pi(B)\|$\\
\item[($iii$)] $\|B_\infty-B_k\|<c^k\|B-\pi(B)\|$
\end{itemize}
\end{theorem}

\begin{proof}
Let $c_6$ and $\epsilon_6$ in Theorem \ref{t1} be given by $c_6=c$
and
\begin{equation}\label{apgf}\epsilon_6=({1-c})\epsilon/2.\end{equation}
Let $s_6$ be given by Theorem \ref{t1} and pick
\begin{equation}\label{kiu}s<\frac{s_6(1-\epsilon)}{4(2+\epsilon)}\end{equation} such that
$\pi(\B(A,s))\subset\B(A,s_6/4)$, (recall that
$\epsilon\leq\epsilon_2<1$, by assumption). The latter condition
ensures that
\begin{equation}\label{oi} \|\pi(B)-A\|<s_6/4.\end{equation} Let
$l=\|B-\pi(B)\|$ and note that
\begin{equation}\label{eq654}l\leq\|B-A\|+\|A-\pi(B)\|\leq s+s_6/4.\end{equation} First note that
$\|B_0-B\|=\|\pi_{1}(B)-B\|\leq \|\pi(B)-B\|= l$ and that
$\pi(B)=\pi(B_0)$ by Lemma \ref{l3}, so
$$\|B_0-\pi(B_0)\|\leq \|B_0-B\|+\|B-\pi(B)\|\leq 2l.$$ Applying
Theorem \ref{t1} we get
$$\|B_{k+1}-\pi(B_{k+1})\|\leq\|B_{k+1}-\pi(B_k)\|\leq  c\|B_{k}-\pi(B_k)\|,$$
as long as
\begin{equation}\label{afg}B_k\in\B(A,s_6).\end{equation} Assuming
this for the moment we get
\begin{equation}\label{agt} \|B_{k}-\pi(B_{k})\|\leq
2lc^k\end{equation} and (Theorem \ref{t1} and Lemma \ref{l3})
\begin{equation}\label{agt1} \|\pi(B_{k+1})-\pi(B_{k})\|\leq
\epsilon_6(2lc^k).\end{equation} The sequence
$(\pi(B_k))_{k=1}^\infty$ is thus a Cauchy sequence, and hence
converges to some point $B_\infty$. By (\ref{agt}) the sequence
$(B_k)_{k=1}^\infty$ must also converge, and the limit point is
again $B_\infty$, which thus satisfies $B_{\infty}=\pi(B_\infty)$
since $\pi$ is continuous. By the triangle inequality, the fact
$\pi(B)=\pi(B_0)$, (\ref{apgf}) and (\ref{agt1}) we have
$$\|\pi(B_k)-\pi(B)\|<\frac{2\epsilon_6 l}{1-c}<\epsilon l,$$ and
combining this with (\ref{oi}), (\ref{eq654}), (\ref{agt}) we also
have \begin{align*}&\|A-B_k\|\leq
\|A-\pi(B)\|+\|\pi(B)-\pi(B_k)\|+\|\pi(B_k)-B_k\|<s_6/4+\epsilon
l+2l\leq \\&\leq s_6/4+\epsilon
(s+s_6/4)+2(s+s_6/4)<s_6,\end{align*} where the last inequality
follows by (\ref{kiu}). With these estimates at hand, it is easy to
turn the above argument into a proper induction proof in which
(\ref{afg}) is verified at each step. We omit the details.
\end{proof}

\section{Rank $k$ matrices versus Hankel matrices; numerical examples}\label{numexp}
We now continue the example in Section \ref{I} and the application
to approximation of a signal by sums of $k$ exponential functions,
as outlined towards the end. By Theorem \ref{t3}, an arbitrary
intersection point of $\r_k$ and $\H$ is almost surely
non-tangential. Given the signal $f$ we let $g$ be the closest ``sum
of $k$ exponentials'' to $f$ in the $\|\cdot\|_w$-norm, see
(\ref{eq53}). Theorem \ref{main} thus says that if $\|f-g\|_w$ is
not too large, the sequence of alternating projections will converge
to an $H(g_\infty)$, where $g_{\infty}\approx g$ is a sum of $k$
exponentials. More precisely, let $$\K=\left\{\sum_{j=1}^k
c_j(\alpha_j^l)_{l=0}^{2n-2}:~c_j\in\C,~\alpha_j\in\C\right\}\subset
\C^{2n-1}.$$ Then for every $\epsilon>0$ there exists an $s>0$ such
that if $\dist(f,\K)<s$, then
$$\|g_\infty-g\|_w<\epsilon ~\dist(f,\K).$$

We now discuss what happens if $f$ is not close enough to $g$ that
Theorem \ref{main} can be applied for any $\epsilon$. Since both
$\pi_{\H}$ and $\pi_{\r_k}$ are contractions, the sequence of
alternating projections $(B_j)_{j=0}^\infty$ (with $B_0=H(f)$) will
be bounded. Thus it has a convergent subsequence, and the limit
point is easily seen to be in $\H_k$. It is not hard to deduce that
$\dist(B_j,\H_k)\rightarrow 0$ as $j\rightarrow\infty$. The only way
the whole sequence could avoid converging is thus if it switches
endlessly along the valleys of the thin set of tangential points,
thereby avoiding the surrounding hills made by the open sets
$\cup_{A\in \H_k}\B(A,s_A)$, where $s_A>0$ is given by Theorem
\ref{main} (with $\epsilon=1$ say) for all non-tangential $A$'s, and
$s_A=0$ for the thin set of tangential $A$'s. That this could happen
seems highly unlikely to us, and we have certainly never encountered
it in practice. However, we leave it as an open problem to prove
that this can not occur.

For practical purposes, it is of course of interest that $s$ be a
relatively large number, (given some fixed $\epsilon$ and $c$).
Inspecting the proof Theorem \ref{main}, one gets the impression
that $B(A,s)$ will not be distinguishable even with binoculars. To
test the actual relationship between $s$ and $\epsilon$ in Theorem
\ref{main}, (with $c=1$), we conducted a few experiments which we
now present. A function\footnote{In this section we will sloppily
say function when we talk of its sampling, i.e. a sequence.} $g$ is
written as a sum of $10$ exponential functions, and then ${f}$ is
generated by adding a smaller function $n$ to $g$. $g$ is normalized
such that $\|g\|_{w}=1$, (recall (\ref{eq53})), $n$ is chosen such
that $H(n)$ is orthogonal to $\H_{10}^n$ at $Hg$, (with respect to
the Hilbert-Schmidt norm). By Proposition \ref{p2}, $H(g)$ is likely
very close to $\pi(H(f))$, and so we will estimate
$\dist(H(f),\H_{10})=\|H(f)-\pi(H(f))\|$ by
$\|H(f)-H(g)\|=\|n\|_{w}$ and $\|H(f_\infty)-\pi(H(f))\|$ by
$\|H(f_\infty)-H(g)\|$. Note that there is no explicit way to compute
$\pi(H(f))$. We set
$$\varepsilon(g,n)=\frac{\|H(f_\infty)-H(g)\|}{\|H(n)\|}=\frac{\|f_\infty-g\|_{w}}{\|n\|_{w}}.$$
Let $s$ be a parameter taking values $0.1, 0.01,~0.001,~0.0001$ and
$0.00001$. For each value of $s$ we randomly generate 50 normalized
$g$'s as above, and for each $g$ we randomly generate a $n$ with
$\|n\|_{w}=s$. The supremum of $\varepsilon(g,n)$ should give us an
idea of what value of $\epsilon$ allows for such an $s$ as in
Theorem \ref{main}. The results are demonstrated in Figure
\ref{figbrusvitt} and Figure \ref{figbursexp}. The difference between the two is that
Figure \ref{figbrusvitt} used white
noise, and in Figure \ref{figbursexp}, the ``noise'' consisted of sums of randomly generated
exponentials. Judging from these figures, $\epsilon\approx\sigma$
seems to be a good rule of thumb, although this will of course vary
substantially from one application to another. The method seems to work slightly better for the case of white noise.


\begin{figure}[tbh!]
\unitlength=1in
\begin{center}
\includegraphics[width=4.5in]{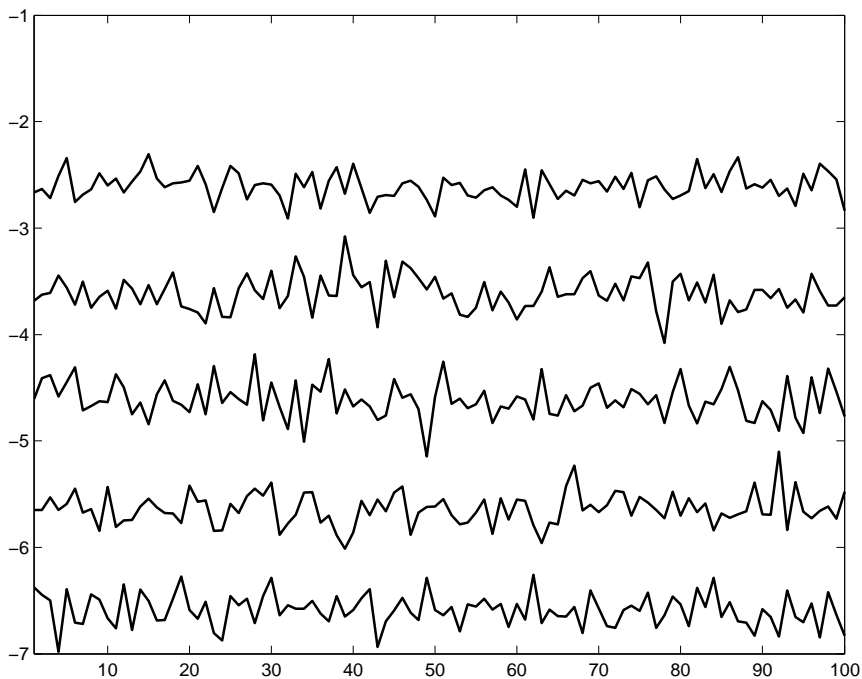}
\end{center}
\caption{\label{figbrusvitt} Signal with 10 exponentials and white noise.}
\end{figure}%

\begin{figure}[tbh!]
\unitlength=1in
\begin{center}
\includegraphics[width=4.5in]{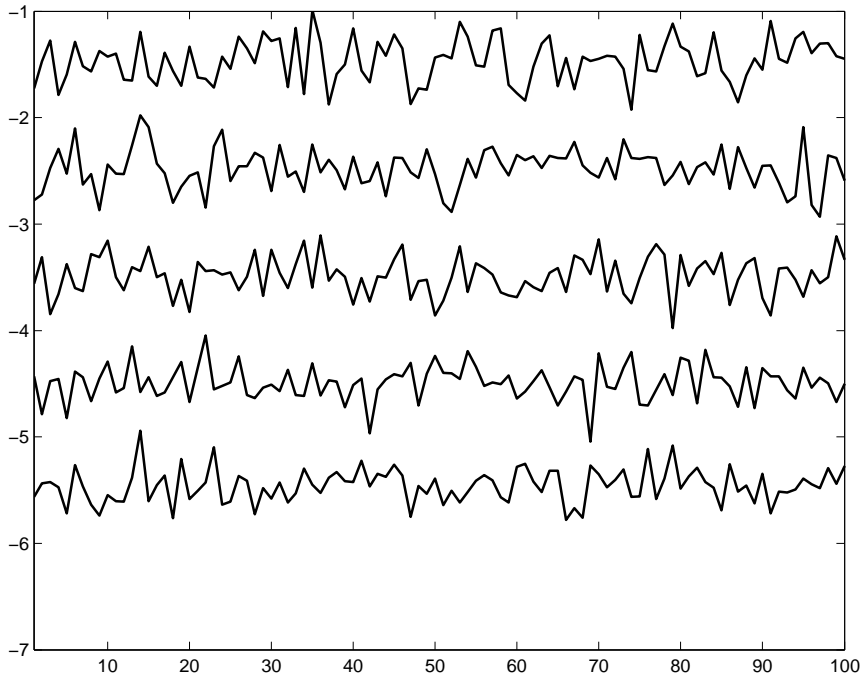}
\end{center}
\caption{\label{figbursexp} Signal with 10 exponentials and sums of other exponentials added.}
\end{figure}%

\section{Summary and open problems}
We have developed a theoretical framework that enables us to
understand convergence properties of alternating projection schemes. This provides substantially
stronger results than previously available, e.g., for theorems
relying solely on Zangwill's Global Convergence Theorem, which
neither provide convergence rate estimates nor information about how
far away the computed approximation is from the optimal value.
Moreover, in contrast to the theory developed for the case of
transversal manifolds, our framework provides more information under
very non-restrictive requirements on the manifolds.

We end by noting a few open problems. We first discuss smoothness.
By examples similar to Example \ref{ex2}, it is not hard to see that
one needs the manifolds to be at least $C^1$ for the alternating
projections to converge. By inspection of the proofs, one easily
sees that for Theorem \ref{main} to hold, it suffices with
$C^1_{Lip}$, i.e. that $d\phi$ is Lipschitz continuous. It is thus an
open question whether Theorem \ref{main} is true assuming only $C^1$. Next we discuss non-tangentiality. Again, Example \ref{ex2} can be tailored such that the alternating projections does not converge, despite assuming $C^\infty$, if we allow the manifolds to be tangential. However, for all applications we are aware of, the manifolds are algebraic, and it seems quite unlikely that a similar thing could happen in this case. We conjecture that for algebraic manifolds, the sequence of alternating projections will always converge, without any additional assumptions. Recall Example \ref{ex3}, where the sequence clearly converges albeit \textit{extremely slowly}. For practical purposes, non-tangentiality is thus still vital.

\section{Acknowledgements}
This work was supported by the Swedish Research Council and the Swedish Foundation for International Cooperation in Research and Higher Education.

\textit{Acknowledgement: This work was partially conducted at Lund
University, supported by the Swedish Research Council; Grant
2008-23883-61232-34}

\bibliographystyle{amsplain}
\bibliography{alternating_projections_bibliography}

\end{document}